\documentclass[12pt]{amsart}
\usepackage{amscd, fullpage}

\usepackage{fancyhdr}
\usepackage{lineno}

\newtheorem{Theorem}{Theorem}
\newtheorem{Lemma}[Theorem]{Lemma}
\newtheorem{Corollary}[Theorem]{Corollary}

\newtheorem{Proposition}[Theorem]{Proposition}

\newtheorem{Definition}[Theorem]{Definition}

\newtheorem{Remark}[Theorem]{Remark}

\def\Vol{\mbox{Vol}}

\def\Z{{\mathbb Z}}

\def \re{{\mathbb R}}

\def \0{\lambda_{0}}

\def \E{\widetilde{E}}

\def\h{{\rm h}_{\rm top}(g)}

\def\Fs{{\mathcal F}}
\def\Ts{{\mathcal T}}

\def\is{\rm{Iso}}

\def \E{\mathbb{E}}

\usepackage{graphicx, amssymb, amsfonts}

\begin{document}
\title{Perelman's invariant and collapse via geometric characteristic splittings}

\author[P. Su\'arez-Serrato]{}

\address{Mathematisches Institut Ludwig-Maximilians-Universit\"at M\"unchen, Theresienstrasse 39, M\"unchen 80333 Germany.}
\email{Pablo.Suarez-Serrato@mathematik.uni-muenchen.de}

\maketitle
\begin{center}
by  Pablo Su\'arez-Serrato
\\[.5cm]
\today
\end{center}
\begin{abstract}{
 Any closed orientable and smooth non-positively curved manifold $M$ is known to admit a geometric characteristic splitting, analogous to the JSJ decomposition in three dimensions. We show that when this splitting consists of pieces which are Seifert fibered or pieces each of whose fundamental group has non-trivial centre, $M$ collapses with bounded curvature and has zero Perelman invariant.}
\end{abstract}

\maketitle

\section{Introduction}
In this paper we study the relationship between the geometric characteristic splitting of any compact connected non-positively curved manifold $M$ as described by B. Leeb and P. Scott in \cite{LS} and the possibility that $M$ collapses with bounded curvature   in the sense of J. Cheeger and M. Gromov in \cite{CG,Gro, PP}. We show how this information can be used to prove vanishing results about G. Perelman's $\bar{\lambda}$ invariant, which was introduced in \cite{P}.

A smooth manifold $M$ is said to {\it collapse with bounded curvature} if there exists a sequence of metrics $\{ g_{i} \} $ for which the sectional curvature is uniformly bounded, but whose volumes tend to zero as $i$ tends to infinity. Cheeger and Gromov have shown that if a manifold $M$ admits a generalised torus action, technically known as a polarised $\Fs$-structure, then $M$ collapses with bounded curvature \cite{Gro, CG}.

Consider a smooth orientable compact and connected  $n$-manifold $M$ with non-positive curvature and convex boundary. It has been shown by Leeb and Scott that $M$ admits a geometric decomposition analogous to the topological Jaco-Shalen-Johanson \cite{JS, J} torus decomposition in dimension three. Either $M$ has a flat metric, or $M$ can be decomposed along totally geodesic codimension one submanifolds which are flat in the metric induced from $M$. The resulting pieces of this decomposition are either Seifert fibered or codimension-one atoroidal.

One of the aims of this paper is to understand how far the analogy with the JSJ decomposition holds. We will show that some asymptotic invariants, described below, vanish on codimension-one atoroidal pieces whose fundamental groups have non-trivial centre. This allows us to further distinguish between the codimension-one atoroidal pieces.

As pointed out in \cite{LS} it may not be necessary to require that $M$ be non-positively curved for such a decomposition to exist in that E. Rips and Z. Sela have shown in \cite{RS} that an algebraic counterpart holds true, (under some mild but technical hypothesis) any finitely presented group admits a JSJ splitting.


Our main result is

\begin{Theorem}{Let $M$ be closed orientable and smooth non-positively curved $n$-manifold. Assume the pieces of the geometric characteristic splitting of $M$ are Seifert fibered or that the centre of the fundamental group of each piece is non-trivial. Then $M$ admits a polarised $\Fs$-structure.}
\end{Theorem}


The existence of a polarised $\Fs$-structure on $M$ implies that the minimal volume of $M$ is zero \cite{CG} and therefore also implies that the minimal entropy of $M$ is zero \cite{PP}. Here the minimal entropy is the infimum of $\h$, the topological entropy of the geodesic flow of $(M,g)$, as we run over all unit volume smooth metrics $g$ on $M$.

As $M$ is non-positively curved, if it is not flat then its fundamental group contains a free non-abelian subgroup \cite{Ball} and therefore any smooth metric $g$ on $M$ will have $\h >0$. So the minimal entropy problem cannot be solved for an $M$ as in Theorem 1; the infimum ${\rm h} (M)$ will never be attained by any smooth metric on $M$.

Examples of manifolds for which the minimal entropy problem will be solved are quotients of Euclidean space with a flat metric $g$, which has $\h =0$, or quotients of Hyperbolic space with a constant negative curvature metric. The latter case was shown by G. Besson, G. Courtois and S. Gallot in a series of papers which found deep connections and consequences in the theory, the interested reader is invited to consult \cite{B,BCG}.

The minimal entropy ${\rm h} (M)$ is related to the minimal volume $ {\rm MinVol}(M)$, volume entropy $\lambda(M)$ and simplicial volume $||M||$ of $M$ in the following string of inequalities, noticed by M. Gromov, A. Manning and others \cite{Gro, Man, BCG, Pat}
\[ \frac{n^{n/2}}{n!}||M|| \leq \lambda(M)^{n} \leq  {\rm h} (M)^{n} \leq (n-1)^{n} {\rm MinVol}(M). \]

If $M$ were flat then $M$ would already admit an $\Fs$-structure, since by Bieberbach's theorem $M$ is finitely covered by $T^{n}$. On the other hand, if the sectional curvatures $K_{M}$ of $M$ were negative at every point of $M$, then $||M|| \neq 0$ by a theorem of W.P. Thurston \cite{Gro} which was expanded by H. Inoue and K. Yano \cite{IY}, so $M$ cannot admit $\Fs$-structures. Therefore the interesting scenario is precisely when $K_{M} \leq 0$, $M$ is not flat and does not admit a metric with strictly  negative curvature.

The existence of an $\Fs$-structure on $M$ also implies that the Yamabe invariant (or \emph{sigma constant}) of $M$ is non-negative \cite{LeB, PP}. The Yamabe invariant of $M$ is positive if and only if $M$ admits a metric of positive scalar curvature \cite{Sc}. According to Gromov and B. Lawson any compact smooth manifold which carries a metric of non-positive sectional curvature cannot carry a metric of positive scalar curvature \cite{GL}. Therefore for any compact smooth manifold which admits both a metric of non-positive sectional curvature and an $\Fs$-structure the Yamabe invariant  vanishes. We also obtain information about Perelman's $\bar{\lambda}$ invariant, because in this case they coincide \cite{K2,AIL}.
\begin{Corollary} Assume $M$ is a compact smooth non-positively curved manifold whose geometric characteristic splitting consists of pieces which are Seifert fibered or pieces whose fundamental group has non-trivial centre, then
 \[ \bar{\lambda}(M)=0.\] \end{Corollary}

Therefore the vanishing of $\bar{\lambda}(M)$ detects a certain lack of hyperbolicity in the sense that neither $M$, or the pieces of its geometric characteristic splitting, may admit a smooth metric of negative sectional curvature.

As a consequence of the description of the characteristic flat submanifold of $M$ in \cite{LS} and Gromov's Cutting Off Theorem \cite{Gro}, we obtain an estimate for the simplicial volume of $M$.

\begin{Proposition} Let $M$ be any closed orientable non-positively curved manifold. Let $N$ denote the complement in $M$ of the  pieces of the geometric characteristic splitting which are Seifert fibered and pieces each of whose fundamental group has non-trivial centre, then
$ ||M|| =  ||N||.$ \end{Proposition}

Perhaps the most intriguing of the invariants mentioned above is presently ${\rm h}(M)$, as it is not yet known whether it is homotopy invariant or if it depends on the differentiable structure of $M$. It should be pointed out that the simplicial volume and the volume entropy are homotopy invariant \cite{Bab, Br}, whereas the minimal volume was shown by L. Bessi\`eres to be sensitive to changes in the differentiable structure of $M$ \cite{Bes}. In fact even the vanishing of the minimal volume is not an invariant of topological type, as shown by D. Kotschick \cite{Kot}.


This paper is organised as follows: section 2 contains the relevant definitions and the proofs of Theorem 1 and Proposition 3 are found in the third section. Finally in the fourth section we present a simple example of a manifold whose characteristic geometric splitting has only codimension-one atoroidal pieces and which admits a polarised $\Fs$-structure (in fact a polarised $\Ts$-structure). This example suggests a possible strengthening of the codimension-one atoroidal condition so that we may distinguish between pieces of such type for which the asymptotic invariants mentioned above vanish and those for which they are positive. This point of view also follows the analogy with the JSJ decomposition. Because in dimension 3 the only codimension-one atoroidal pieces are the hyperbolic pieces, which have positive simplicial volume. It also raises a natural question: is the minimal volume of a non-positively curved manifold which has a completely atoroidal piece in its characteristic geometric splitting positive?

{\bf Acknowledgements:} The author wishes to thank Dan Jane for detailed comments and suggestions on a previous version. The author is grateful for comments  Dieter Kotshick made on a previous version and thanks Gabriel Paternain, Bernhard Leeb and G\'erard Besson for interesting conversations. Special thanks also to Hartmut Wei{\ss}    for explaining the relevance of enlargeability and to Michael Brunnbauer for type-setting help. The author is supported by the  Deutsche Forschungsgemeinschaft under the project `Asymptotic Invariants of Manifolds'.

\section{ Definitions}

\subsection{The geometric characteristic splitting}

 Being consistent with \cite{LS}, we recall the following

\begin{Definition}{ A manifold $N$ of dimension $n$ is {\bf Seifert fibered} if $N$ is a Seifert bundle over a 2-dimensional orbifold with fiber a flat $(n-2)$-manifold.}
\end{Definition}

So $N$ is foliated by $(n-2)$-dimensional closed flat manifolds, each leaf $F$ has a foliated neighbourhood $U$ which has a finite cover whose induced foliation is a product $F \times D^2$.

\begin{Definition}{A manifold $N$ of dimension $n$ is {\bf codimension-one atoroidal} if any $\pi_1$-injective map of a $(n-1)$-torus into $M$ is homotopic into the boundary of $N$.}
\end{Definition}

\subsection{ $\Fs$-structures}
An $\Fs$-structure is a generalisation of an $S^1$-action. The existence of an $\Fs$-structure on a manifold implies some of its asymptotic invariants vanish \cite{Gro, CG, PP}.

\begin{Definition}{ An $\Fs$-structure on a closed manifold $M$ is given by,
\begin{enumerate}
\item{ A finite open cover $\{ U_1, ..., U_{N} \} $;}
\item{ $\pi_{i}:\widetilde{U}_{i}\rightarrow U_{i}$ a finite Galois covering with group of deck transformations $\Gamma_{i}$, $1\leq i \leq N$;}

\item{ A smooth torus action with finite kernel of the $k_{i}$-dimensional torus, \\ $\phi_{i}:T^{k_{i}}\rightarrow {\rm{Diff}}(\widetilde{U}_{i})$, $1\leq i \leq N$;}

\item{ A homomorphism $\Psi_{i}:\Gamma_{i}\rightarrow {\rm{Aut}}(T^{k_{i}})$ such that
\[ \gamma(\phi_{i}(t)(x))=\phi_{i}(\Psi_{i}(\gamma)(t))(\gamma x) \]
for all $\gamma \in \Gamma_{i}$, $t \in T^{k_{i}}$ and $x \in \widetilde{U}_{i}$; }

\item{ For any finite sub-collection $\{ U_{i_{1}}, ..., U_{i_{l}} \} $ such that  $U_{i_{1}\ldots i_{l}}:=U_{i_{1}}\cap \ldots \cap U_{i_{l}}\neq\emptyset$ the following compatibility condition holds: let $\widetilde{U}_{i_{1}\ldots i_{l}}$ be the set of points $(x_{i_{1}}, \ldots , x_{i_{l}})\in \widetilde{U}_{i_{1}}\times \ldots \times \widetilde{U}_{i_{l}}$ such that $\pi_{i_{1}}(x_{i_{1}})=\ldots = \pi_{i_{l}}(x_{i_{l}})$. The set $\widetilde{U}_{i_{1}\ldots i_{l}}$ covers $\pi_{i_{j}}^{-1}(U_{i_{1}\ldots i_{l}}) \subset \widetilde{U}_{i_{j}}$ for all $1\leq j \leq l$, then we require that $\phi_{i_{j}}$ leaves $\pi_{i_{j}}^{-1}(U_{i_{1}\ldots i_{l}})$ invariant and it lifts to an action on $\widetilde{U}_{i_{1}\ldots i_{l}}$ such that all lifted actions commute. }
\end{enumerate}}\end{Definition}

An $\Fs$-structure is said to be {\em pure} if all the orbits of all actions at a point, for every point have the same dimension.

We will say an $\Fs$-structure is {\em polarised} if the smooth torus action $\phi_{i}$ above are fixed point free for every $U_{i}$.

\subsection{$\Fs$-structures on flat manifolds}\label{flat}

The isometry group of $\E^n$ is the semidirect product of $\re^n$ and $ {\rm O}(n)$.
Let $\rho:  {\rm O}(n)\to \mbox{\rm Aut}(\re^n)$ be the map $\rho(B)(x)=Bx$.
Let $\Gamma\subset \is(\E^n)$ be a cocompact lattice and $M:=\E^n/\Gamma$ a compact flat manifold.
Let $p:\Gamma\to  {\rm O}(n)$ be the homomorphism $p(t,\alpha)=\alpha$, where $(t,\alpha)\in \re^n\times  {\rm O}(n)$.
The Bieberbach theorem ensures that $\Gamma$ meets the translations in a lattice
( the kernel of $p$ is isomorphic to $\Z^n$) and $p(\Gamma)$ is a finite group $G$.
Then $M$ is finitely covered by the torus $\re^n/\mbox{\rm ker}(p)$ and the deck transformation group of this finite
cover is $G$. 

Notice that for any $\alpha\in G$, $\rho(\alpha)$ maps $\mbox{\rm ker}(p)$ to itself
because
\[(u,\alpha)\circ (s,I)\circ (u,\alpha)^{-1}=(\rho(\alpha)s,I)\]
and thus if $(s,I)\in\Gamma$, then $(\rho(\alpha)s,I)\in \Gamma$.

It follows that the  map $\rho:  {\rm O}(n)\to \mbox{\rm Aut}(\re^n)$ induces a map
$$\psi:G\to \mbox{\rm Aut}( T^n=\re^n/\mbox{\rm ker}(p)).$$

As an action $\phi$ of $T^n$ on
$\re^n/\mbox{\rm ker}(p)$ we take $x\mapsto x+t$.
To see that this defines an $\Fs$-structure we check
the condition $\alpha(\phi(t)(x))=\phi(\psi(\alpha)(t))(\alpha (x))$ for $\alpha\in G$ which just says
$\alpha(x+t)=\alpha(x)+\alpha(t)$.

This $\Fs$-structure on $M$ extends to an $\Fs$-structure on the product $M\times I$, where $I$ denotes any interval.

\subsection{Asymptotic Invariants}

The simplicial volume $||M||$ of a closed orientable manifold $M$ is defined as the infimum of $\Sigma_{i}|r_{i}|$ where $r_{i}$ are the coefficients of any \emph{real} cycle representing the fundamental class of $M$. For the definition and relevant properties in case $M$ has boundary, the reader is invited to consult \cite{Gro}.

For a closed connected smooth Riemannian manifold $(M,g)$, let $\rm{Vol}(M,g)$ denote its volume and let $K_{g}$ its sectional curvature. We define the following minimal volumes \cite{Gro}:

$$\rm{MinVol}(M):=\inf\limits_{g} \{ \rm{Vol}(M,g)\quad : \quad |K_{g}| \leq 1 \}$$
and
$$\rm{Vol_{K}}(M):=\inf\limits_{g} \{ \rm{Vol}(M,g)\quad : \quad K_{g} \geq -1 \}.$$ 
\medskip

The vanishing of  $\rm{Vol_{K}}(M)$ implies that the simplicial volume of $M$ is also zero, using Bishop's comparison theorem. If $M$ admits an $\Fs$-structure then  $\rm{Vol_{K}}(M)=0$ \cite{PP}.

The minimal entropy ${\rm h}(M)$ of a closed smooth manifold $M$ is defined as the infimum of the topological entropy $\h$ of the geodesic flow of $g$ over the family of $C^{\infty}$ Riemannian metrics $g$ on $M$ with unit volume.

The geometric meaning of $\h$ is best expressed by Ma\~n\'e's formula \cite{Man, Pat}
$$\h= \lim_{T\rightarrow
\infty}\frac{1}{T}\log \int_{M\times M}n_{T}(p,q)\;dp\,dq. $$
Here $n_{T}(p,q)$ is defined to be the number of
geodesic arcs of length $\leq T$ joining the points $p$ and $q$ of $M$. So that if $\h >0$ then there are on average exponentially many geodesics between any two points of $M$.

\subsection{ Perelman and Yamabe invariant(s)}

Assume $(M^{n}, g)$ is a smooth compact manifold of dimension $n \geq 3$. G. Perelman considered in \cite[pg. 7]{P} the smallest eigenvalue $\lambda_{g}$ of the elliptic operator $4 \Delta_{g} + s_{g}$, here $\Delta = d^{\ast}d= -\nabla \cdot \nabla$ is the positive-spectrum Laplace-Beltrami operator associated with $g$ and $s_{g}$ is the scalar curvature of $g$. Perelman defined
$$ \bar{\lambda}(M):= \sup\limits_{g} \lambda_{g} \Vol(M,g)^{2/n}.$$
Where the supremum is taken over all smooth metrics $g$ on $M$.

Amongst his various remarkable contribuitions to the understanding of the Ricci flow, Perelman observed that $\bar{\lambda}$ is non-decreasing along the Ricci flow whenever it is non-positive. A detailed proof can be found in \cite[8.I.2.3, p. 22]{KL1}.

Now we review the definition of the Yamabe invariant. Consider a fixed conformal class of metrics $\gamma $ on the smooth closed manifold $M$, and let the Yamabe constant of $(M, \gamma )$ be
$$ \mathcal{Y}(M, \gamma )= \inf\limits_{g \in \gamma} \frac{\int\limits_{M} s_{g} dvol_{g}}{(\Vol(M,g))^{2 / n }}.$$

The {\emph{Yamabe invariant}} is then defined to be
$$ \mathcal{Y}(M) = \sup\limits_{\gamma} \mathcal{Y}(M, \gamma ),$$
where the supremum is taken over all conformal classes of metrics on $M$.

A result of G. Paternain and J. Petean \cite[Theorem 7.2]{PP} states that if a smooth compact manifold $M$ admits an $\Fs$-structure then $ \mathcal{Y}(M) \geq 0$.

It was noted by D. Kotschick in \cite{K2} that the Perelman and Yamabe invariants essentially coincide. The precise relationship between them was shown by A. Akutagawa, M. Ishida and C. LeBrun in \cite{AIL} to be
$$\bar{\lambda}(M) = \begin{cases}
\mathcal{Y}(M) & \text{if } \quad \mathcal{Y}(M) \leq 0 \\
+\infty & \text{if } \quad \mathcal{Y}(M) > 0.
\end{cases}$$

On the other hand, a well known fact about the Yamabe invariant is that $\mathcal{Y}(M) >0 $ if and only if $M$ admits a smooth metric of positive sectional curvature, see for example \cite{Sc}. Therefore when $M$ admits an $\Fs$-structure and cannot admit a metric of positive scalar curvature, we get $\mathcal{Y}(M) = \bar{\lambda}(M) =0$.

This is precisely the case for non-positively curved manifolds, as mentioned in the introduction, since Gromov and Lawson have shown in \cite{GL} that non-positively curved manifolds can not carry a metric of positive scalar curvature. So when a non-positively curved manifold admits an $\Fs$-structure, we have $\mathcal{Y}(M) = \bar{\lambda}(M) =0$.

In particular, applied to the construction in Theorem 1, the above discussion provides a proof of Corollary 2.

\begin{Remark} The same argument shows $\bar{\lambda}(M) =0$ when $M$ admits an $\Fs$-strucure, $M$ is enlargeable in the sense of \cite{GL} and its universal covering space is spin becuase under these conditions $M$ cannot admit a metric of positive scalar curvature.
\end{Remark}

\section{Proofs}

We begin by recalling some aspects of the description of the geometric characteristic splitting which we will use later, for consistency we will keep the notation in \cite{LS} throughout. The reader may consult \cite{Ball} for facts about non-positively curved manifolds and \cite{LS} for further details about the characteristic geometric splitting.

Let $X$ be a simply connected non-positively curved manifold. For every isometry $\phi$ of $X$ denote by ${\rm MIN}(\phi)$ the  set where $d_{\phi}: x\to d(x, \phi x)$ assumes its infimum. If $A$ is an abelian subgroup of the isometry group of $X$, define
$${\rm MIN}(A):= \bigcap\limits_{\phi\in A} {\rm MIN}(\phi)\cong E \times Y.$$
Here $E$ is a Euclidean space and $Y$ is a simply connected manifold of non-positive curvature with convex boundary.

Let $C(A)$ be the centraliser of $A$ and $N(A)$ its normaliser. That $N(A)$ acts on $X$ preserving the metric splitting  ${\rm MIN}(A)\cong E \times Y$ is shown in \cite{LS}. A flat is called $A$-invariant if the action of any element of $A$ fixes it as a set. We will call a flat submanifold $F$ of $X$ a $\Gamma$-flat if the action of $\Gamma$ on $F$ has compact quotient.

We will first treat the case when the geometric characteristic splitting of $M$ consists of Seifert fibered pieces only.

\begin{Theorem}{Let $M$ be a closed non-positively curved smooth $n$-manifold. Assume $M$ only has Seifert fibered pieces in its geometric characteristic splitting. Then $M$ admits a pure polarised $\Fs$-structure.}
\end{Theorem}

\begin{proof} Let $N$ be one of the Seifert fibered pieces of the geometric characteristic splitting of $M$. Then $N$ is diffeomorphic to the Seifert fibered manifold $S_{A}$, here $A$ is the abelian subgroup of $\Gamma =\pi_1(N)$ which defines the fibration as in \cite{LS}. Recall $S_{A}= H_{A} / N(A)$, where $H_{A}$ is the closed convex hull of all $A$-invariant $\Gamma$-flats and $N(A)$ is the normaliser of $A$.

For a closed convex subset $Z$ of $Y$ we have $H_{A}=Z \times E \subset Y\times E = {\rm MIN}(A)$. Bieberbach's theorem implies that on $E \times Y$ the abelian group $A$ acts by translations on the Euclidean factor $E$.

The group $\Gamma $ acts by deck transformations on the universal covering $\tilde{H}_{A}=\tilde{Z}\times E$ of $H_{A}$. Therefore ${\is}(\tilde{H}_{A}) = {\is}(\tilde{Z}) \times {\is}( E ) $. So (by a slight abuse of notation) we can view $\Gamma $ as a subgroup of ${\is}(\tilde{H}_{A})$, the isometry group of $\tilde{H}_{A}$. Say ${\rm dim} (E)=k$ and consider the projection homomorphism
$${\is}(\tilde{H}_{A}) = {\is}(\tilde{Z}) \times {\is}( E ) \rightarrow {\is}(E)\to {\rm O}(k).$$
Then we have a homomorphism $\Gamma \to {\rm O}(k)$ with image a finite group $G$. Its kernel is a finite index subgroup $\Gamma_{0} \subset {\is}(\tilde{Z}) \times \re^{k}$.
It follows that the manifold
$$N_{0}:=\tilde{H}_{A} / \Gamma_{0}\cong (\tilde{Z}\times E)/\Gamma_{0}\cong (\tilde{Z} / \Gamma_0)\times T^{k}$$ is a finite cover of $N$ with $G$ as a deck transformation group.
In this way we obtain an $\Fs$-structure on $N$, with $\Psi:G\to {\rm Aut}(T^{k})$ as in \ref{flat}, to comply with all the requirements in the $\Fs$-structure definition.
As we are assuming every piece of the geometric characteristic splitting of $M$ is Seifert fibered, this construction furnishes $M$ with an $\Fs$-structure.

 This $\Fs$-structure is pure because the dimension of the tori which define the local actions is $n-2$ over each Seifert piece and can be taken to be $n-2$ over neighbourhoods of the flat hypersurfaces. It is also polarised; $T^{n-2}$ acts freely on itself.\end{proof}

We now consider the case when the pieces of the geometric characteristic splitting have fundamental groups with non-trivial centre, this situation is slightly more general than the Seifert fibered case. However, the strategy is the same; each piece will be finitely covered by a smooth manifold which splits off a torus. The resulting $\Fs$-structure, although polarised, will not always be pure because the dimension of these tori may vary from piece to piece.

\begin{Theorem}
Let $M$ be a closed smooth orientable manifold of non-positive sectional curvature. Assume that the fundamental group of every piece of its geometric characteristic splitting has non-trivial centre. Then $M$ admits a polarised $\Fs$-structure.
\end{Theorem}

\begin{proof} Let $N$ denote a piece of the geometric characteristic splitting. By the main result of P. Eberlein in \cite{Eb2} if the centre of $\pi_1(N)$ is non-trivial, then there exists a finite covering $N^0$ of $N$ such that $N^0$ is diffeomorphic to $N^{\ast} \times T^{k}$. Therefore $N^0 \rightarrow N$ induces a polarised $\Fs$-structure on $N$ as in the previous Theorem. This is true for every piece of the geometric characteristic splitting, by hypothesis.

Let $S$ denote a flat hypersurface of $M$ which is a component of the geometric characteristic flat manifold $V$ found in the geometric splitting. Choose a small enough $\epsilon$-neighbourhood $N(S)$ of $S$ in $M$, so that $N(S)$ is diffeomorphic to $S \times (-\epsilon , \epsilon ) $ and is disjoint from other such flat characteristic hypersurface components. Then $N(S)$ admits an $\Fs$-structure induced from the flat $\Fs$-structure on $S$ which was described in \ref{flat}.

All the above $\Fs$-structures commute on overlaps and are polarised, therefore they endow $M$ with a polarised $\Fs$-structure.
\end{proof}

Theorem 1 now follows from Theorems 8 and 9.

 Next we see that the simplicial volume does not detect the pieces of the geometric characteristic splitting which are Seifert fibered or have fundamental groups with non-trivial centre.

\subsection{Proof of Proposition 3}
\begin{proof}
Let $V$ denote the characteristic flat submanifold of $M$ which determines its geometric splitting. Recall that the fundamental group of each flat component of $V$ injects into $\pi_1(M)$. Also note that fundamental groups of flat manifolds are amenable, so that we can now cut $V$ off from $M$, as in \cite{Gro}, and its simplicial volume will remain unaffected. Let $N_{i}$ denote the components of $M - V$ and $N$ denote the complement in $M$ of the Seifert pieces and of the pieces each of whose fundamental group has non-trivial centre. Then the above discussion implies that
$$ ||M|| = ||M - V || = \sum_{i}||N_{i}||.$$

If the piece $N_{i}$ is Seifert fibered or its fundamental group has non-trivial centre, $N_{i}$ admits an $\Fs$-strucutre and therefore $||N_{i}||=0$. So that by Gromov's Cutting Off Theorem (in reverse, pasting $V$ back to $N$) we have $ ||M|| =  ||N||$.
\end{proof}

\subsection{Complete atoroidality }

Consider a hyperbolic 3-manifold $M$ of finite volume and with geodesic boundary $\partial M$, such that $\partial M$ has only one component and this is a 2-torus. Endow $N=M\times S^{1}$ with the product of the hyperbolic metric on $M$ and the standard metric on $S^1$, so that $N$ is a non-positively curved manifold. Denote by $DN$ the double of $N$, two copies of $N$ glued along $\partial N = T^3$ by the identity; $DN$ is by construction non-positively curved.

The unique flat 3-manifold which can be embedded in $DN$ is $\partial N = T^3$. So the characteristic geometric splitting (by construction) of $DN$ is precisely the two $N$ pieces separated by $\partial N = T^3$. Notice that the two copies of $N$ constitute the pieces of the characteristic splitting and they are codimension-one atoroidal. Any map of $T^3$ into $N$ will be parallel to $\partial N = T^3$. The action of $S^{1}$ on itself in $N=M\times S^{1}$ on each piece is fixed point free. Further, the actions on each piece commute along the boundary $\partial N = T^3$. Therefore $DN$ admits a polarised $\Fs$-structure (actually a polarised $\Ts$-structure) and therefore ${\rm MinVol}(M)=0$ and all the asymptotic invariants of $M$ vanish.

This example is codimension-1 atoroidal, but not codimension-2 atoroidal. Take a geodesic loop $\gamma$ in $M$ whose class in $\pi_1(M)$ is non-trivial. Then $\gamma \times S^1$ is an embedding of $T^2$ into $M\times S^1$ which is not boundary parallel. In fact, it is clear that the above example can be extended to arbitrary dimensions. Simply by taking $M$ a hyperbolic $(n-1)$-manifold of finite volume with only one cusp and repeating the construction. Even so, Theorem 1 does apply since the centre of $\pi_1(M)$ is non-trivial.

We conclude that the codimension-one atoroidal pieces are not quite the precise $n$-dimensional analogy of the hyperbolic pieces of the JSJ decomposition in dimension 3.

\begin{Definition}{ An $n$-dimensional smooth manifold $M$ is said to be {\bf codimension-$k$ atoroidal} if any $\pi_1$-injective map of a flat $(n-k)$-torus into $M$ is homotopic into the boundary of $M$.}
\end{Definition}

In fact, the examples above suggest that maybe the pieces which would have positive simplicial volume, or at least some positive asymptotic invariants, should be {\em codimension-k atoroidal} for all $k$ between $1$ and $n-2$. This discussion leads naturally to the following

\begin{Definition}{ An $n$-dimensional smooth manifold is said to be {\bf completely atoroidal} if it is codimension-$k$ atoroidal for any $k$ with $1\leq k \leq n-2$. }
\end{Definition}

The relationship between complete atoroidality and the fundamental group can be seen in the following lemma, which follows from the well known results of D. Gromoll and J. Wolf \cite{GW} and of Lawson and S.T. Yau \cite{LY}.

\begin{Lemma}
If the compact smooth manifold $N$ is completely atoroidal and carries a metric of non-positive curvature, then the centre of $\pi_1(N)$ is trivial.
\end{Lemma}

\begin{proof}
 We observe the contrary is known. Denote by $Z(\pi_1(N))$ the centre of $\pi_1(N)$. The group $Z(\pi_1(N))$ is Abelian, so $Z(\pi_1(N))\neq \{ e \}$ implies by \cite{GW} \& \cite{LY} that there exists a torus in $N$ carrying it. Therefore $N$ is not completely atoroidal.
\end{proof}

This last lemma explains the range of Theorem 1 and leads us to wonder about the nature of completely atoroidal pieces. Assume that the closed smooth and non-positively curved manifold $M$ has a completely atoroidal piece in its geometric characteristic splitting, is ${\rm MinVol}(M)$ positive ?




\begin{thebibliography}{aa}


\bibitem{AIL} K. Akutagawa, M. Ishida, C. LeBrun, {\it Perelman's Invariant, Ricci Flow, and the Yamabe Invariants of Smooth Manifolds}, Arch. Math. 88 (2007) no.1, 71-76.

\bibitem{Bes} L. Bessi\`{e}res, {\it Un th\'{e}or\`{e}me de rigidit\'{e} diff\'{e}rentielle}, Comm.Math.Helv. 73 (1998) 443-479.

\bibitem{B} G. Besson, {\it Minimal entropy and Mostow's rigidity theorems}, Erg. Theo. \& Dyn. Sys. 16 (1996) 623-649.

\bibitem{BCG} G. Besson, G. Courtois and S. Gallot,{\it Volume et entropie minimale des espaces localement symmetriques}, Invent.Math. 103 (1991) 417-445.

\bibitem{Bab} I.K. Babenko, {\it Asymptotic Invariants of Smooth Manifolds}, Izv. Ros. Akad. Nauk. Ser. Math. 56 (1992) 707-751, (russian); engl. trans. in Russian Acad. Sci. Izv. Math. 41 (1993) 1-38.

\bibitem{Ball} W. Ballman, {\it Lectures on spaces of non-positive curvature}, DMV Seminar Bd. 25, Birkh\"auser, Berlin, 1995.

\bibitem{Br} M. Brunnbauer, {\it Homological Invariance for Asymptotic Invariants and Systolic Inequalities}, to appear in GAFA, Preprint (2007), {\tt arxiv:math.GT/07020789}.

\bibitem{CG} J. Cheeger and M. Gromov, {\it Collapsing Riemannian Manifolds while
keeping their curvature bounded I}, Jour.Diff.Geom. 23 (1986)
309-346.

\bibitem{Eb1} P. Eberlein, {\it A canonical form for compact nonpositively curved manifolds whose fundamental groups have nontrivial center}, Math. Ann. 260 (1982) no. 1, 23-29.

\bibitem{Eb2} P. Eberlein, {\it  Euclidean de Rham factor of a lattice of nonpositive curvature}, J.Differential Geom.  18  (1983) no. 2, 209-220.

\bibitem{FM} A. Freire and R. Ma\~n\'e, {\it On the entropy of the geodesic flow in manifolds without conjugate points}, Invent. Math. 69 (1982), no. 3, 375-392.


\bibitem{GW} D. Gromoll and J.A. Wolf, {\it Some relations between the metric structure and the algebraic structure of the fundamental group in manifolds of nonpositive curvature},
Bull. Amer. Math. Soc. 77 1971 545-552.

\bibitem{Gro} M. Gromov, {\it Volume and Bounded Cohomology}, Pub. Math.
I.H.E.S. tome 56 (1982) 5-99.

\bibitem{GL} M. Gromov and H.B. Lawson, {\it Positive scalar curvature and the Dirac operator on complete riemannian manifolds}, Pub. Mat. I.H.E.S. 58 (1983) 83-196.

\bibitem{IY} H. Inoue and K. Yano, {\it The Gromov invariant of negatively curved manifolds}, Topology Vol. 21 No.1 (1981) 83-89.

\bibitem{JS} W. Jaco and P. B. Shalen, {\it Seifert fibered spaces in 3-manifolds}, Mem. AMS 220 (1979).

\bibitem{J} K. Johanson, {\it Homotopy equivalences of 3-manifolds with boundary}, Springer LNM 761 (1979).

\bibitem{KL} M. Kapovitch and B. Leeb, {\it Actions of discrete groups on non-positively curved spaces}, Math. Ann. 306 (1996) No.2, 341-352.

\bibitem{KL1} B. Kleiner and J. Lott, {\it Notes on Perelman's Papers}, Preprint (2007),\\
{\tt arXiv:math.DG/0605667v2}.

\bibitem{Kot} D. Kotschick, {\it Entropies, Volumes and Einstein Metrics}, Preprint (2004),\\ {\tt arxiv:math.DG/0410215}.

\bibitem{K2} D. Kotschick, {\it  Monopole classes and Perelman's invariant of four-manifolds}, Preprint (2006),  {\tt  arXiv:math.DG/0608504}.

\bibitem{LY} H.B. Lawson and S.T. Yau, {\it Compact manifolds of nonpositive curvature},
J. Diff. Geom. 7 (1972), 211--228.

\bibitem{LS} B. Leeb and P. Scott, {\it A geometric characteristic splitting in all dimensions,} Comm. Mat. Helv. 75 (2000) 201-215.

\bibitem{LeB} C. LeBrun, {\it Kodaira dimension and the Yamabe problem}, Comm. An. Geom. 7 (1999) 133-156.

\bibitem{Man} R. Ma\~{n}\'{e}, {\it On the topological entropy of geodesic flows},
J. Diff. Geom. {\bf 45} (1997) 74--93.

\bibitem{Pat} G. Paternain, {\it Geodesic Flows}, Progress in Mathematics 180, Birkh\"{a}user (1999).

\bibitem{PP} G. Paternain and J. Petean, {\it Minimal Entropy and collapsing with
curvature bounded from below}, Invent.Math. 151 (2003) 415-450.

\bibitem{PP1} G. Paternain and J. Petean, {\it Entropy and collapsing of compact complex surfaces},  Proc. London Math. Soc. (3)  89  (2004),  no. 3, 763-786.

\bibitem{P} G. Perelman, {\it The entropy formula for the Ricci flow and its geometric applications}, Preprint (2002), {\tt arXiv:math.DG/0211159}.


\bibitem{Ra} M.S. Raghunathan, {\it Discrete subgroups of Lie groups}, Springer-Verlag,  New York (1972).

\bibitem{RS} E. Rips and Z. Sela, {\it Cyclic splittings of finitely presented groups and the canonical JSJ decomposition}, Ann. of Math. 146(2) (1997), 53-109.


\bibitem{Sc} R. Schoen, {\it Variational Theory for the Total Scalar Curvature Functional for Riemannian Metrics and Related Topics}, LNM 1365, Springer Verlag, Berlin (1987) 120-154.




 \end{thebibliography}

\end{document}